\documentclass[10pt]{amsart}

\usepackage[pagebackref]{hyperref}
\hypersetup{
    colorlinks = true,
    citecolor = [rgb]{0.00, 0.5, 0.00},
}

\renewcommand*{\backref}[1]{}
\renewcommand*{\backrefalt}[4]{%
    \ifcase #1 (Not cited.)%
    \or        (Cited on page~#2.)%
    \else      (Cited on pages~#2.)%
    \fi}

\usepackage{mathtools}

\usepackage[margin=1.175in]{geometry}
\usepackage{todonotes}
\usepackage{multirow}
\usepackage{longtable}

\usepackage{amsmath}
\usepackage{amsfonts}
\usepackage{amssymb}
\usepackage{float}
\usepackage{amsthm}
\usepackage{comment}
\usepackage{amscd}
\usepackage{amsxtra}
\usepackage{epsfig}

\usepackage{listings}
\usepackage{color}

\definecolor{dkgreen}{rgb}{0,0.6,0}
\definecolor{gray}{rgb}{0.5,0.5,0.5}
\definecolor{mauve}{rgb}{0.58,0,0.82}

\usepackage{color}
\usepackage[all]{xy}
\usepackage{verbatim}

\usepackage{eucal}
\usepackage{mathrsfs}

\usepackage{graphicx}
\usepackage{tikz-cd}

\usepackage{url}
\usepackage{verbatim}

\usepackage{xy}
\xyoption{all}

\usepackage{amsthm}

\usepackage{caption} 
\makeatletter
\def\@tocline#1#2#3#4#5#6#7{\relax
  \ifnum #1>\c@tocdepth 
  \else
    \par \addpenalty\@secpenalty\addvspace{#2}%
    \begingroup \hyphenpenalty\@M
    \@ifempty{#4}{%
      \@tempdima\csname r@tocindent\number#1\endcsname\relax
    }{%
      \@tempdima#4\relax
    }%
    \parindent\z@ \leftskip#3\relax \advance\leftskip\@tempdima\relax
    \rightskip\@pnumwidth plus4em \parfillskip-\@pnumwidth
    #5\leavevmode\hskip-\@tempdima
      \ifcase #1
       \or\or \hskip 1em \or \hskip 2em \else \hskip 3em \fi%
      #6\nobreak\relax
    \hfill\hbox to\@pnumwidth{\@tocpagenum{#7}}\par
    \nobreak
    \endgroup
  \fi}
\makeatother

\DeclareMathOperator{\Ext}{Ext}

\DeclareMathOperator{\spec}{Spec}

\DeclareMathOperator{\QCoh}{\mathrm{QCoh}}

\DeclareMathOperator{\Sym}{Sym}

\newtheorem{lemma}{Lemma}[section]

\newtheorem{corollary}[lemma]{Corollary}

\newtheorem{theorem}[lemma]{Theorem}

\newtheorem{prop}[lemma]{Proposition}
\newtheorem{conjecture}[lemma]{Conjecture}

\theoremstyle{definition}

\newtheorem{definition}[lemma]{Definition}

\newtheorem{remark}[lemma]{Remark}

\newcommand{\FF}{\mathbf{F}}
\newcommand{\Z}{\mathbf{Z}}

\newcommand{\cC}{\mathcal{C}}

\newcommand{\RR}{\mathbf{R}}

\newcommand{\M}{{M}}
\newcommand{\Mfg}{\M_{FG}}
\newcommand{\Mcub}{\M_\mathrm{cub}}
\newcommand{\Mell}{\M_\mathrm{ell}}
\newcommand{\sm}{\mathrm{sm}}

\newcommand{\co}{\mathcal{O}}
\newcommand{\ce}{\mathcal{E}}

\newcommand{\PP}{\mathbf{P}}

\newcommand{\GG}{\mathbb{G}}

\newcommand{\Mod}{\mathrm{Mod}}

\newcommand{\Map}{\mathrm{Map}}

\newcommand{\Lone}{L_{K(1)}}

\newcommand{\der}{\mathrm{der}}
\newcommand{\Alg}{\mathrm{Alg}}

\newcommand{\Eoo}{{\mathbf{E}_\infty}}

\newcommand{\cf}{\mathcal{F}}
\newcommand{\DD}{\mathbf{D}}

\newcommand{\dR}{\mathrm{dR}}

\newcommand{\ol}[1]{\overline{#1}}

\newcommand{\E}[1]{\mathbf{E}_{{#1}}}
\newcommand{\mmod}{/\!\!/}

\renewcommand{\S}{\mathbb{S}}

\newcommand{\GL}{\mathrm{GL}}

\newcommand{\MU}{\mathrm{MU}}

\newcommand{\tmf}{\mathrm{tmf}}

\newcommand{\TMF}{\mathrm{TMF}}

\newcommand{\BPP}{\mathrm{BP}}

\renewcommand{\H}{\mathrm{H}}

\newcommand{\BSpin}{\mathrm{BSpin}}

\newcommand{\MString}{\mathrm{MString}}

\title{Hodge theory for elliptic curves and the Hopf element $\nu$}
\author{Sanath Devalapurkar}
\email{sanathd@mit.edu}

\begin{document}

\maketitle

\begin{abstract}
    We show that the vector bundle on the moduli stack $M_\mathrm{ell}$ of
    elliptic curves associated to the $2$-cell complex $C\nu$ is isomorphic to
    the de Rham cohomology sheaf
    $\mathrm{H}^1_\mathrm{dR}(\mathcal{E}/M_\mathrm{ell})$ of the universal
    elliptic curve $\mathcal{E}\to M_\mathrm{ell}$.  We use this to calculate
    the homotopy groups of the $\mathbf{E}_{1}$-quotient $\mathrm{tmf} /\!\!/
    \nu$ of $\mathrm{tmf}$ by $\nu$, called the spectrum of ``topological
    quasimodular forms'', by relating its Adams--Novikov spectral sequence to
    the cohomology of the moduli stack of cubic curves with a chosen splitting
    of the Hodge--de Rham filtration.
\end{abstract}

\section{Introduction}

In this paper, we study the relationship between the Hopf invariant one element
$\nu\in \pi_3 \tmf$ and the Hodge filtration for elliptic curves. Namely, we
show that the vector bundle on the moduli stack $\Mell$ of elliptic curves
associated to $C\nu$ is isomorphic to the (middle) de Rham cohomology
$\H^1_\dR(\ce/\Mell)$ of the universal elliptic curve $\ce\to \Mell$. A version
of this relationship had been stated by Hopkins in \cite[Section
5]{hopkins-icm}. Using this, we calculate the homotopy groups of the
$\E{1}$-quotient $\tmf\mmod\nu$ of $\tmf$ by $\nu\in \pi_3(\S)$ by showing that
the $E_2$-page of its Adams--Novikov spectral sequence is isomorphic to the
cohomology of the moduli stack of cubic curves with a chosen splitting of the
Hodge--de Rham filtration. The $\E{1}$-ring $\tmf\mmod\nu$ is called the
spectrum of \emph{topological quasimodular forms} (see Remark
\ref{name-quasimodular}).

The ring spectrum $\tmf\mmod\nu$ is interesting for several reasons. One
motivation for studying it comes from the Ando-Hopkins-Rezk orientation
$\MU\langle 6\rangle \to \tmf$ (see \cite{koandtmf}). As is made clear during
the course of the proof, a key reason for why this orientation does not factor
through the map $\MU\langle 6\rangle\to \mathrm{MSU}$ is because $\tmf$ detects
the element $\nu \in \pi_3(\S)$; this in turn is related to the fact that the
weight $2$ Eisenstein series is not a modular form. Since $\tmf\mmod\nu$ is the
``smallest'' coherently structured (i.e., $\E{1}$-) $\tmf$-algebra with a
nullhomotopy of $\nu$, one might expect the composite
\begin{equation}\label{comp}
    \MU\langle 6\rangle \to \tmf\to \tmf\mmod\nu
\end{equation}
to factor through $\mathrm{MSU}$ via an $\E{1}$-map; see also Remark
\ref{orientations}. Although we do not prove in this paper that the composite
\eqref{comp} factors through $\mathrm{MSU}$, we will use the results of this
paper to address this question in future work. The connection between $\nu$ and
the weight $2$ Eisenstein series is also discussed in Section
\ref{descent-sseq}.  The relationship between $\nu$ and de Rham cohomology is
also intrinsically intriguing, because the Hodge--de Rham filtration on the de
Rham cohomology of an elliptic curve is related to many deep topics in
artihmetic geometry (such as Grothendieck--Messing theory).

We begin in Section \ref{hodge-background} by recalling some background on Hodge
theory for cubic curves from algebraic geometry. In particular, we give a Hopf
algebroid presentation for the moduli stack $\Mcub^\dR$ of cubic curves with a
chosen splitting of the Hodge--de Rham exact sequence. In Section
\ref{tmf-A-hodge}, we prove our main technical result relating the
Adams--Novikov spectral sequence of $\tmf\mmod\nu$ to the cohomology of the
moduli stack $\Mcub^\dR$.  Finally, in Section \ref{descent-sseq} we prove
Theorem \ref{hodge-theorem}, which calculates this Adams--Novikov spectral
sequence. It degenerates at the $E_4$-page, and $\TMF\mmod\nu$ is found to be
$24$-periodic. Moreover, $\tmf\mmod\nu$ is homotopy commutative. These results
were discovered independently by Rezk in unpublished work, and we give our own
proof of his calculation of $\pi_\ast(\tmf\mmod\nu)$.

\subsection{Acknowledgements}

I would like to thank Charles Rezk: after we proved part of Theorem
\ref{hodge-theorem}, we discovered that he had proved the result independently;
I would like to thank him for discussions about the results in this paper, and
for letting us write up this result. I am also grateful to Mark Behrens, Robert
Burklund, Tyler Lawson, Lennart Meier, and Andrew Senger for helpful
discussions, and in particular to Robert Burklund for providing helpful comments
on a draft.

\section{Background on Hodge theory}\label{hodge-background}

In this section, we recall some background on Hodge theory for cubic curves
over a general base scheme. Multiple sources (such as \cite[Appendix
A1.2]{katz-p-adic}) discuss Hodge theory for (smooth) elliptic curves.

Let $f:X\to Y$ be a morphism of schemes. One then has the $f^{-1} \co_Y$-linear
relative de Rham complex $\Omega^\bullet_{X/Y}$.
\begin{definition}
    The \emph{$i$th relative de Rham cohomology} $\H^i_\dR(X/Y)$ of $f:X\to Y$
    is defined to be the hypercohomology sheaf $\RR^i
    f_\ast(\Omega^\bullet_{X/Y})$ on $Y$.
\end{definition}
The Tot spectral sequence defines the Hodge--de Rham spectral sequence of
sheaves on $Y$:
$$E_1^{s,t} = R^t f_\ast \Omega^s_{X/Y} \Rightarrow \H^{s+t}_\dR(X/Y).$$
The following (easy) result is well-known.
\begin{theorem}
    If $f:X\to Y$ is a smooth, proper, and surjective morphism of relative
    dimension $1$ with geometrically connected fibers, then the Hodge--de Rham
    spectral sequence degenerates at the $E_1$-page.
\end{theorem}
Since $f$ is of relative dimension $1$, the only interesting de Rham cohomology
is in the middle dimension, i.e., $\H^1_\dR(X/Y)$. In particular, for such $f$,
there is an exact sequence
$$0\to f_\ast \Omega^1_{X/Y} \to \H^1_\dR(X/Y)\to R^1 f_\ast \co_X\to 0$$
of quasicoherent sheaves on $Y$; this is called the \emph{Hodge--de Rham exact
sequence}. Moreover, the pairing $\H^1_\dR(X/Y)\otimes_{\co_Y} \H^1_\dR(X/Y)\to
\co_Y$ is determined by the canonical perfect pairing
$$R^1 f_\ast \co_X\otimes_{\co_Y} f_\ast \Omega^1_{X/Y} \to R^1 f_\ast
\Omega^1_{X/Y} \xrightarrow{\mathrm{trace}} \co_Y.$$
We now specialize to the case when $f:X\to Y$ is an elliptic curve $f:E\to
S$. Then $f_\ast \Omega^1_{E/S}$ is the line bundle $\omega_{E/S}$ of invariant
differentials. The pairing $\omega_{E/S}\otimes_{\co_S} R^1 f_\ast \co_E\to
\co_S$ is perfect, and so $R^1 f_\ast \co_E \cong \omega_{E/S}^{-1}$. In
particular, the Hodge--de Rham exact sequence for $E\to S$ becomes
\begin{equation}\label{hdr}
    0\to \omega_{E/S} \to \H^1_\dR(E/S)\to \omega_{E/S}^{-1}\to 0.
\end{equation}
If $\Mell$ denotes the moduli stack of elliptic curves, and $\ce\to \Mell$ is
the universal elliptic curve, then \eqref{hdr} exhibits $\H^1_\dR(\ce/\Mell)$ as
an element of $\Ext^1_{\Mell}(\omega^{-1}, \omega)$.
\begin{remark}
    If $S$ is a $p$-adic scheme, then \eqref{hdr} corresponds to the Hodge
    filtration of the Dieudonn\'e module $\DD(E[p^\infty]/S)$ of $E$ under the
    isomorphism $\H^1_\dR(E/S) \cong \DD(E[p^\infty]/S)$; see, for instance,
    \cite[Section V]{katz-crystalline}.
\end{remark}

The following is an immediate consequence of \cite[Equation
A1.2.3]{katz-p-adic}:
\begin{prop}\label{katz-prop}
    If $f:E\to S$ is an elliptic curve, then there is an isomorphism
    $\H^1_\dR(E/S) \otimes \omega_{E/S} \cong f_\ast
    \Omega^1_{E/S}(\infty)^{\otimes 2}$.
\end{prop}
The map 
$$f_\ast \Omega^1_{E/S}(\infty)^{\otimes 2} \cong \H^1_\dR(E/S) \otimes
\omega_{E/S} \to \omega_{E/S}^{-1}\otimes \omega_{E/S} = \co_S$$
induced by the Hodge--de Rham exact sequence sends a section of $f_\ast
\Omega^1_{E/S}(\infty)^{\otimes 2}$ to its residue at $\infty$.

We can now generalize the above story to the non-smooth setting. First, we
recall the definition of a cubic curve.
\begin{definition}
    A \emph{cubic curve} $f:E\to S$ over a scheme $S$ is a flat and proper
    morphism of finite presentation whose fibers are reduced, irreducible curves
    of arithmetic genus $1$, along with a section $\infty:S\to E$ whose image is
    contained in the smooth locus $E^\sm$ of $f$. Let $\Mcub$ denote the stack
    of cubic curves, and let $f:\ce\to \Mcub$ denote the universal cubic curve.
\end{definition}
Let $\omega$ denote the line bundle on $\Mcub$ assigning to a cubic curve
$f:E\to S$ the cotangent bundle $\omega_{E/S}$ along the section $\infty$. The
vector bundle $\H^1_\dR(\ce/\Mell)$ over $\Mell$ extends to a vector bundle,
denoted $\H^1_\dR(\ce/\Mcub)$, over $\Mcub$; this vector bundle is represented
by an element in $\Ext^1_{\Mcub}(\omega^{-1}, \omega) \cong \H^1(\Mcub;
\omega^{\otimes 2})$. Upon tensoring $\H^1_\dR(\ce/\Mcub)$ with $\omega$, we
obtain for every cubic curve $f:E\to S$ an exact sequence
\begin{equation}\label{hdr-redux}
    0\to \omega_{E/S}^{\otimes 2} \to f_\ast \Omega^1_{E/S}(\infty)^{\otimes 2}
    \to f_\ast \co_{E/S} \cong \co_S\to 0,
\end{equation}
which we shall refer to as the Hodge--de Rham exact sequence.
\begin{definition}
    Let $\Mcub^\dR$ denote the moduli stack of cubic curves with a chosen
    splitting of the Hodge--de Rham exact sequence, and let $\Mell^\dR =
    \Mcub^\dR\times_{\Mcub} \Mell$ denote the moduli stack of elliptic curves
    with a chosen splitting of the Hodge--de Rham exact sequence.
\end{definition}
In order to do calculations with $\Mcub^\dR$, we would like to obtain a Hopf
algebroid presentation of this stack. To do so, we recall a Hopf algebroid
presentation of $\Mcub$. Zariski-locally on any base scheme $S$, a cubic curve
is described by a Weierstrass equation
\begin{equation}\label{weierstrass}
    y^2 + a_1 xy + a_3 y = x^3 + a_2 x^2 + a_4 x + a_6,
\end{equation}
with other choices of coordinates $(x,y)$ given by the transformations
$$x\mapsto x+r, \ y\mapsto y + sx + t.$$
The moduli stack $\Mcub$ of cubic curves is presented by the Hopf algebroid
$$(D,\Gamma) = (\Z[a_1, a_2, a_3, a_4, a_6], D[r,s,t]),$$
with gradings $|a_i| = i$ and $|r| = 2$, $|s| = 1$, and $|t| = 3$. Studying how
the coefficients $a_i$ transform gives the right unit $\eta_R:D\to \Gamma$ of
this Hopf algebroid:
\begin{align*}
    a_1 & \mapsto a_1 + 2s,\\
    a_2 & \mapsto a_2 - a_1 s + 3r - s^2,\\
    a_3 & \mapsto a_3 + a_1 r + 2t,\\
    a_4 & \mapsto a_4 + a_3 s + 2a_2 r - a_1 t - a_1 rs - 2st + 3r^2,\\
    a_6 & \mapsto a_6 + a_4r - a_3 t + a_2 r^2 - a_1 rt - t^2 + r^3.
\end{align*}

To determine a Hopf algebroid presentation of $\Mcub^\dR$, note that the
coordinate $x$ defines a function on $\ce$ with double pole at $\infty$, and
that it is in fact the only such non-constant function on the smooth locus of
$\ce$ (this follows from the usual calculation \cite[Section 2.2.5]{katz-mazur}
with the Riemann-Roch formula). In particular, Proposition \ref{katz-prop}
implies that the element $[r]$ in the cobar complex for the Hopf algebroid
$(D,\Gamma)$ detects the extension in $\Ext^1_{\Mcub}(\co, \omega^{\otimes 2}) =
\Ext^{1,2}_\Gamma(D, D)$ determined by the de Rham cohomology
$\H^1_\dR(\ce/\Mell)$. By the exact sequence \eqref{hdr-redux}, a choice of
Hodge--de Rham splitting on the universal cubic curve amounts to fixing a choice
of coordinate $x$ (although $y$ is allowed to vary).  Consequently:
\begin{prop}\label{de-rham-presentation}
    The moduli stack $\Mcub^\dR$ of cubic curves with a chosen splitting of the
    Hodge--de Rham exact sequence is presented by the Hopf algebroid $(D,\Sigma)
    = (\Z[a_1, a_2, a_3, a_4, a_6], D[s,t])$, with gradings\footnote{Recall that
    the topological grading is \emph{double} the algebraic grading.} $|a_i| =
    i$, $|s| = 1$, and $|t| = 3$. The right unit is the same as in that of the
    elliptic curve Hopf algebroid, except with $r=0$:
    \begin{align}
        a_1 & \mapsto a_1 + 2s, \nonumber\\
        a_2 & \mapsto a_2 - a_1 s - s^2, \nonumber\\
        a_3 & \mapsto a_3 + 2t, \nonumber\\
        a_4 & \mapsto a_4 + a_3 s - a_1 t - 2st, \nonumber\\
	a_6 & \mapsto a_6 - a_3 t - t^2. \label{right-unit}
    \end{align}
\end{prop}

\section{The relationship with $\tmf \mmod \nu$}\label{tmf-A-hodge}

In this section, we study the $\E{1}$-quotient $\tmf\mmod\nu$ of $\tmf$ by
$\nu$, and relate its Adams--Novikov spectral sequence to the cohomology of
$\Mcub^\dR$.

We begin by recalling one construction of the $\E{1}$-quotient $\tmf\mmod\nu$.
This satisfies the following universal property: if $R$ is any
$\E{1}$-$\tmf$-algebra, then
$$\Map_{\Alg_{\E{1}}(\Mod(\tmf))}(\tmf\mmod\nu, R) = \begin{cases}
    \Omega^{\infty+4} R & \text{if }\nu = 0\in \pi_3 R\\
    \emptyset & \text{else}.
\end{cases}$$
One of the main results of \cite{barthel-thom} justifies the following
definition:
\begin{definition}
    The $\E{1}$-ring $\tmf\mmod\nu$, called \emph{topological quasimodular
    forms} (see Remark \ref{name-quasimodular} for a justification for the
    name), is the Thom spectrum of the dotted extension in the following
    diagram:
    $$\xymatrix{
	S^4 \ar[r]^-\nu \ar[d] & B\GL_1(\tmf)\\
	\Omega S^5 \ar@{-->}[ur] & 
    }$$
\end{definition}
\begin{remark}\label{thom-base}
    The element $\nu\in \pi_3(\tmf)$ is spherical, and so this diagram factors
    as
    $$\xymatrix{
	S^4 \ar[r]^-{2v_1^2} \ar[d] & \BSpin \ar[r]^-J & B\GL_1(\S) \ar[d]\\
	\Omega S^5 \ar@{-->}[ur] & & B\GL_1(\tmf).
    }$$
    Following the notation of \cite{tmf-witten}, we will write $A$ to denote the
    Thom spectrum of the loop map $\Omega S^5\to B\GL_1(\S)$. Then there is a
    canonical equivalence $\tmf\mmod\nu \simeq \tmf \otimes A$ of
    $\E{1}$-$\tmf$-algebras, so there is in particular a ring map $A\to
    \tmf\mmod\nu$. 
\end{remark}
The $\E{1}$-ring $A$ will be useful below. In \cite{tmf-witten}, it is shown
that there is an $\E{1}$-map $A\to \BPP$. Moreover, the $\BPP$-homology of $A$
at the prime $2$ is isomorphic to $\BPP_\ast[y_2]$, where $y_2$ is sent to
$t_1^2$ modulo decomposables under the map $\BPP_\ast(A)\to \BPP_\ast(\BPP)$. In
particular, $\H_\ast(A;\FF_2) \cong \FF_2[\zeta_1^4]$. One then has:
\begin{prop}\label{A-bp}
    There is a nontrivial simple $2$-torsion element $\sigma_1\in \langle \eta,
    \nu, 1_A\rangle \subseteq \pi_5(A) \cong \pi_5(C\nu)$ specified up to
    indeterminacy by the relation $\eta\nu = 0$.  One choice of this element is
    represented by $[t_2]$ in the Adams--Novikov spectral sequence for $A$, and
    by $h_{21}$ in the (mod $2$) Adams spectral sequence for $A$.
\end{prop}

To connect $\tmf\mmod\nu$ and Hodge theory for cubic curves, we make the
following observation. Recall that $\H^1_\dR(\ce/\Mcub)\in
\Ext^1_{\Mcub}(\omega^{-1}, \omega)$.
\begin{prop}\label{hodge-nu}
    Let $f:\ce\to \Mcub$ denote the universal cubic curve over the moduli stack
    of cubic curves. Then $\H^1_\dR(\ce/\Mcub)\in \Ext^1_{\Mcub}(\omega^{-1},
    \omega) \cong \H^1(\Mcub;\omega^2)$ detects $\nu$ in the $E_2$-page of the
    descent spectral sequence for $\tmf$.
\end{prop}
\begin{proof}
    This is essentially argued in \cite[Section 5.2]{hopkins-icm}. We know that
    $\H^1(\Mcub; \omega^2) = \Z/12$ by the calculations in \cite{bauer-tmf}; the
    element $[r]$ in the cobar complex determined by the Hopf algebroid
    $(D,\Gamma)$ is a representative for the generator. This element detects
    $\nu$ in the Adams--Novikov spectral sequence for $\tmf$, and by the
    discussion before Proposition \ref{de-rham-presentation}, also detects the
    extension class of the Hodge--de Rham exact sequence.
\end{proof}
\begin{corollary}
    The rank two vector bundle $\cf(C\nu)$ on the moduli stack of cubic curves
    corresponding to $C\nu$ is isomorphic to $\H^1_\dR(\ce/\Mcub)$.
\end{corollary}
\begin{remark}\label{twisting}
    In \cite[Section 11.5]{rezk-unpublished}, the Hodge--de Rham exact sequence
    appears in a different but related guise, as a class in the $E_2$-page of a
    spectral sequence $\{E_r^{s,t}\}$ converging to the homotopy groups of the
    space of $\Eoo$-maps $\Z_+ \to \TMF$. The element $\H^1_\dR(\ce/\Mell)\in
    E_2^{1,4}$ detects a nontrivial class in $\pi_3 \Map_{\Eoo}(\Z_+, \TMF) =
    \pi_3 \GG_m(\TMF)$, i.e., an $\Eoo$-map $K(\Z,3)_+ \to \TMF$. This is
    related to the $\Eoo$-twisting of $\TMF$ explored in \cite{twist-tmf}.
\end{remark}

Since $\tmf\mmod\nu$ is the $\E{1}$-quotient of $\tmf$ by $\nu$ by Remark
\ref{thom-base}, it is the universal $\E{1}$-$\tmf$-algebra with a nullhomotopy
of $\nu$. If $\tmf \mmod\nu$ is a homotopy commutative ring (which we will show
is indeed the case in Corollary \ref{tmf-htpy-comm}), then we would be able to
consider the stack associated to $\tmf \mmod\nu$ (in the sense of \cite[Chapter
9]{tmf}, \cite[Section 2.1]{homologytmf}), and it would be reasonable to expect
that Proposition \ref{hodge-nu} implies that this stack is the moduli of cubic
curves with a choice of splitting of the Hodge--de Rham spectral sequence. We
have:
\begin{theorem}\label{nu-theorem}
    Let $g:\Mcub^\dR\to \Mcub$ denote the structure morphism. Then the sheaf on
    $\Mcub$ associated to $A$ is isomorphic as an algebra to the pushforward
    $g_\ast \co_{\Mcub^\dR}$.
\end{theorem}
\begin{proof}
    Let $\cC$ be a presentable symmetric monoidal ($\infty$-)category, and let
    $T$ denote the functor $\cC^\mathrm{unital}\to \Alg_{\E{1}}(\cC)$ sending a
    unital object $i:\mathbf{1}\to X$ to the free $\E{1}$-algebra in $\cC$ whose
    unit factors through $i$: this may be defined via the homotopy pushout
    $$\xymatrix{
	\mathrm{Free}_{\E{1}}(\mathbf{1}) \ar[r] \ar[d] & \mathbf{1} \ar[d]\\
	\mathrm{Free}_{\E{1}}(X) \ar[r] & T(X)
    }$$
    in $\Alg_{\E{1}}(\cC)$. The functor $\cf(-)$ from the homotopy category of
    $\tmf$-modules to $\QCoh(\Mcub)$ fits into a commutative diagram
    $$\xymatrix{
	h\Mod(\tmf)^\mathrm{unital} \ar[r] \ar[d]^-T & h\Alg_{\E{1}}(\Mod(\tmf))
	\ar[d]^-T\\
	\QCoh(\Mcub)^\mathrm{unital} \ar[r] & \Alg_\mathrm{assoc}(\QCoh(\Mcub)).
    }$$
    It is easy to see by the universal property of $T$ that $T(C\nu) \simeq A$,
    so it follows from Proposition \ref{hodge-nu} that $\cf(A) \cong T(f_\ast
    \Omega^1_{\ce/\Mcub}(\infty)^{\otimes 2})$. It therefore suffices to show
    that $T(f_\ast \Omega^1_{\ce/\Mcub}(\infty)^{\otimes 2}) \cong g_\ast
    \co_{\Mcub^\dR}$.  This is not immediate, since the sheaf $T(f_\ast
    \Omega^1_{\ce/\Mcub}(\infty)^{\otimes 2})$ satisfies a universal property
    with respect to {associative} $\co_{\Mcub}$-algebras. Nonetheless, its
    universal property defines an algebra map $\varphi:T(f_\ast
    \Omega^1_{\ce/\Mcub}(\infty)^{\otimes 2}) \to g_\ast \co_{\Mcub^\dR}$ of
    sheaves on $\Mcub$. To check that this is an isomorphism, it suffices to
    show by \cite[Lemma 4.4]{homologytmf} that the map induces an isomorphism on
    the cuspidal cubic over $\FF_p$ for all primes $p$.
    
    In fact, we can check this over $\Z$. We know that $g_\ast \co_{\Mcub^\dR}$
    is isomorphic to a polynomial ring on a single generator: indeed, the space
    of all splittings of the Hodge filtration is isomorphic to the projective
    line $\PP(\H^1_\dR(\ce))$ minus the point corresponding to the line
    determined by the invariant differentials. Moreover, $T(f_\ast
    \Omega^1_{\ce/\Mcub}(\infty)^{\otimes 2})$ is isomorphic to a free
    associative ring on one generator, which is also a polynomial ring on a
    single generator. It therefore suffices to show that the generator maps to
    the generator under the algebra map $\varphi$, but this (after unwinding) is
    precisely the fact (Proposition \ref{katz-prop}) that $\H^1_\dR(\ce/\Mcub)
    \otimes \omega = f_\ast \Omega^1_{\ce/\Mcub} (\infty)^{\otimes 2}$.
\end{proof}
\begin{corollary}\label{nu-cor}
    There is a descent spectral sequence
    $$E_2^{s,2t} = \H^s({\Mcub^\dR}; g^\ast \omega^{\otimes t}) \Rightarrow
    \pi_{2t-s}(\tmf \mmod\nu),$$
    and it is isomorphic to the Adams--Novikov spectral sequence for $\tmf
    \mmod\nu$.
\end{corollary}
\begin{proof}
    There is a descent spectral sequence
    $$E_2^{s,2t} = \H^s({\Mcub}; \cf(A)\otimes_{\co_{\Mcub}} 
    \omega^{\otimes t}) \Rightarrow \pi_{2t-s}(\tmf \mmod\nu),$$
    and this is isomorphic to the Adams--Novikov spectral sequence for $\tmf
    \mmod\nu$ by \cite[Corollary 5.3]{homologytmf} and the main theorem of
    \cite{wood}. Combining Theorem \ref{nu-theorem} with the projection
    isomorphism shows that $\cf(A)\otimes_{\co_{\Mcub}} \omega^{\otimes t} \cong
    g_\ast (g^\ast \omega^{\otimes t})$. The morphism $g$ is flat and affine, so
    $E_2^{s,2t} \cong \H^s({\Mcub^\dR}; g^\ast \omega^{\otimes t})$, as desired.
\end{proof}
\begin{remark}
    Corollary \ref{nu-cor} says that although $\tmf \mmod\nu$ is not \emph{a
    priori} a homotopy commutative ring, there is a descent spectral sequence
    which would exist if there was a sheaf of structured ring spectra on
    $\Mcub^\dR$ whose global sections was $\tmf \mmod\nu$. We pose this as a
    conjecture:
\end{remark}
\begin{conjecture}\label{etale-site}
    For some $k\geq 2$, there is a sheaf of even-periodic $\E{k}$-rings
    $\co^\der$ on the \'etale site of $\Mell^\dR$ such that if $f:\spec R\to
    \Mell^\dR$ is an \'etale map, then $\co^\der(f)$ is the Landweber-exact
    theory corresponding to the composite $\spec R\to \Mell^\dR\to \Mell\to
    \Mfg$, and such that the global sections $\Gamma(\Mell^\dR;\co^\der)$ is
    equivalent as an $\E{1}$-ring to $\TMF\mmod\nu$. Moreover, the resulting
    $\E{k}$-ring structure on $\TMF\mmod\nu$ extends to an $\E{k}$-ring
    structure on $\tmf \mmod\nu$.
\end{conjecture}

\section{The descent spectral sequence}\label{descent-sseq}

Our goal in this section is to calculate the homotopy groups of $\tmf \mmod\nu$
via the descent spectral sequence of Corollary \ref{nu-cor}. To do this
calculation, we will use the Hopf algebroid presentation in Proposition
\ref{de-rham-presentation}.  The calculation of the descent spectral sequence
was done independently by Charles Rezk; although he stated part of the result to
us in an email, the argument is ours (so errors are our fault).
\begin{theorem}\label{hodge-theorem}
    There is an isomorphism
    \begin{equation}\label{e2-anss-hodge}
	\H^\ast(\Mcub^\dR; g^\ast \omega^{\otimes \ast}) \cong \Z[b_2, b_4, b_6,
	b_8, h_1, h_{21}]/I,
    \end{equation}
    where $b_i\in \H^0(\Mcub^\dR; g^\ast \omega^{\otimes i})$ of total degree
    $2i$, $h_1\in \H^1(\Mcub^\dR; g^\ast \omega)$ of total degree $1$, and
    $h_{21}\in \H^1(\Mcub^\dR; g^\ast \omega^{\otimes 3})$ of total degree $5$.
    If one defines
    \begin{align*}
	c_4 & = b_2^2 - 24 b_4, \ c_6 = -b_2^3 + 36 b_2 b_4 - 216 b_6,\\
	\Delta & = -b_2^2 b_8 - 8b_4^3 - 27b_6^2 + 9b_2 b_4 b_6,
    \end{align*}
    then the ideal $I$ is generated by the relations
    $$2h_1=0, \ 2h_{21} = 0, \ h_1^2 b_6 = h_{21}^2 b_2, \ 4b_8 = b_2 b_6 -
    b_4^2, \ 1728 \Delta = c_4^3 - c_6^2.$$
    Moreover, the descent/Adams--Novikov spectral sequence of Corollary
    \ref{nu-cor} collapses on the $E_4$-page, and $\pi_\ast(\tmf\mmod\nu)$ is
    determined by the differentials
    $$d_3(b_2) = h_1^3, \ d_3(b_4) = h_1^2 h_{21}, \ d_3(b_6) = h_1 h_{21}^2, \
    d_3(b_8) = h_{21}^3.$$
    One has the relations
    $$\eta^3 = 0, \ \eta^2\sigma_1 = 0, \ \eta \sigma_1^2 = 0, \ \sigma_1^3 =
    0$$
    in the homotopy of $\tmf\mmod\nu$, in addition to the relations in $I$. Here
    $\eta$ is represented by $h_1$, and $\sigma_1$ is represented by $h_{21}$.
    All the torsion in $\tmf \mmod\nu$ is concentrated in dimensions congruent
    to $1,2\pmod{4}$.
\end{theorem}
Before giving the proof, we discuss some consequences.

\begin{remark}\label{name-quasimodular}
    By Theorem \ref{hodge-theorem}, there is a ring isomorphism
    $$\H^0(\Mcub^\dR; g^\ast \omega^{\otimes \ast})\cong \Z[b_2, b_4, b_6,
    b_8]/(4b_8 = b_2 b_6 - b_4^2, \ 1728 \Delta = c_4^3 - c_6^2).$$
    This ring has been studied before (in characteristic zero, in which case
    $b_8 = (b_2 b_6 - b_4^2)/4$), e.g., in \cite{kaneko-zagier, movasati}, where
    it is referred to as the ring of \emph{quasimodular forms}.  Theorem
    \ref{hodge-theorem} provides a calculation of the ring of integral
    quasimodular forms, and also justifies calling the ring spectrum
    $\tmf\mmod\nu$ by the name ``topological quasimodular forms''.
\end{remark}

The following corollary is a calculation via Theorem \ref{hodge-theorem}.
\begin{corollary}\label{anss-zero}
    In Adams--Novikov filtration zero, $b_i^2$ and twice any monomial in the
    $b_i$s survive to the $E_\infty$-page for $i=2,4,6,8$, as do $b_2 b_6$ and
    $\lambda_1 b_8 b_2^2 + \lambda_2 b_2 b_4 b_6$ for $\lambda_1 \equiv
    \lambda_2\pmod{2}$. In particular, $\Delta\in \pi_{24} (\tmf \mmod\nu)$, so
    the $\E{1}$-ring $\TMF \mmod\nu$ is $24$-periodic with periodicity generator
    $\Delta$.
\end{corollary}
\begin{remark}\label{2-local}
    Note that $\tmf \mmod\nu$ is complex orientable after inverting $2$. This
    can be seen algebraically by noting that the Hopf algebroid $(D,\Sigma)$
    presenting $\Mcub^\dR$ becomes discrete after inverting $2$; indeed, the
    transformation $y\to y - a_1 x/2 - a_3/2$ transforms the Weierstrass
    equation \eqref{weierstrass} into
    $$y^2 = x^3 + a_2 x^2 + a_4 x + a_6,$$
    and one cannot make any coordinate changes to $x$ since it is fixed.  We
    find that $(D, \Sigma)$ is isomorphic to the discrete Hopf algebroid $(D' =
    \Z[1/2][a_2, a_4, a_6], D')$, and so $\pi_\ast(\tmf[1/2] \mmod\nu) \cong
    \Z[1/2][a_2, a_4, a_6]$, with $|a_i| = 2i$. In light of this, we only need
    to prove Theorem \ref{hodge-theorem} after $2$-localization.
\end{remark}
\begin{remark}
    The Hurewicz image of $\tmf$ in $\tmf \mmod\nu$ can be determined from
    Theorem \ref{hodge-theorem}. The subring generated by $\eta$, $\sigma_1$,
    $2b_2$, and $b_2^2$ is in the image of the map $\pi_\ast A\to \pi_\ast
    \tmf\mmod\nu$. There relationship between $\pi_\ast(\tmf \mmod\nu)$ and
    $\pi_\ast(\tmf)$, however, is more interesting than merely the Hurewicz
    image.
    The ($2$-local) calculation in \cite{bauer-tmf} shows that $\ol{\kappa}\nu$
    vanishes in $\pi_{23}(\tmf)$; this is detected in the Adams--Novikov
    spectral sequence by a $d_5$-differential $d_5(\Delta) = \ol{\kappa}\nu$.
    This implies that the element $8\Delta\in \pi_{24}(\tmf)$ can be expressed
    as an element of the Toda bracket $\langle 8, \nu, \ol{\kappa}\rangle$.
    Equivalently, the map $\ol{\kappa}: S^{20} \to \tmf$ extends to a map from
    $\Sigma^{20} C\nu$, and hence from $\Sigma^{20} \tmf \wedge C\nu$.
    Composition with the map $S^{24} \to \Sigma^{20} C\nu$ which is degree $8$
    on the top cell produces the element $8\Delta\in \pi_{24} \tmf$ (up to
    indeterminacy). In other words, $8\Delta$ comes from an element of
    $\pi_{24}(\Sigma^{20} \tmf \wedge C\nu) \cong \pi_4(\tmf \wedge C\nu)$.
    Under the canonical map $\tmf \wedge C\nu\to \tmf \mmod\nu$, this element
    corresponds to $2b_2\in \pi_4(\tmf \mmod\nu)$. Similarly, the element
    $\Delta \eta\in\pi_{24}(\tmf)$ can be related to the element $\sigma_1\in
    \pi_5(\tmf \mmod\nu)$. This is related to the approach taken in
    \cite{tmf-witten} to show that the Ando-Hopkins-Rezk orientation
$\MString\to \tmf$ from \cite{koandtmf} is surjective on homotopy.  \end{remark}
\begin{remark}\label{weight-2-eisenstein} One can give explicit $q$-expansions
    for each of the generators of $\H^0(\Mcub^\dR; g^\ast \omega^{\otimes \ast})
    \cong \Z[b_2, b_4, b_6, b_8]$. The element $b_2$ is related to the weight
    $2$ Eisenstein series via the $q$-expansion: $$b_2 = 1 - 24\sum_{n\geq 1}
    \sigma_1(n) q^n, \ \text{where } \sigma_1(n) = \sum_{d\geq 1, d|n} d.$$
\end{remark} \begin{remark}\label{orientations} The explicit characteristic
    power series for the Witten genus (in \cite[Section 6.3]{hirzebruch}, for
    instance) shows that the Witten genus of a spin $4k$-manifold is a modular
    form if its first Pontryagin class vanishes. This restriction arises
    essentially because of the appearance of a weight $2$ Eisenstein series,
    which is not a modular form, in the characteristic power series. By Remark
    \ref{weight-2-eisenstein}, the weight $2$ Eisenstein series is in fact a
    global section of $\H^0(\Mcub^\dR; g^\ast \omega^{\otimes \ast})$, and so
    one might expect to be able to refine the orientation $\MU\langle
    6\rangle\to \tmf$ from \cite{koandtmf} to an orientation $\mathrm{MSU}\to
    \tmf\mmod\nu$. We will return to this question in forthcoming work; see also
    Remark \ref{twisting}.  \end{remark} The following application is due to
    Rezk.  \begin{corollary}\label{tmf-htpy-comm} The $\E{1}$-ring $\tmf
    \mmod\nu$ admits the structure of a homotopy commutative ring.
    \end{corollary} \begin{proof} We know that $\tmf \mmod\nu$ only has
	$\tmf$-module cells in dimensions divisible by $4$, so all obstructions
	to its homotopy commutativity live in these dimensions. We know that
	these obstructions vanish after inverting $2$, since $A[1/2]$ is
	homotopy commutative. By Theorem \ref{hodge-theorem}, all the homotopy
    groups of $\tmf \mmod\nu$ in dimensions divisible by $4$ are torsion-free,
    so $\tmf \mmod\nu$ must indeed be homotopy commutative.  \end{proof} An
    immediate consequence of Theorem \ref{nu-theorem} and Corollary
    \ref{tmf-htpy-comm} is: \begin{corollary}\label{stack-tmf-A} The stack
	$\M_{\tmf \mmod\nu}$ associated to the homotopy commutative ring $\tmf
	\mmod\nu$ is isomorphic to $\Mcub^\dR$.  \end{corollary} \begin{remark}
	    Corollary \ref{stack-tmf-A} implies, for instance, that the fact
	    that $\nu$ is not detected by $\Lone \tmf$ is related to the
	    existence of a dotted map $$\xymatrix{ & \Mcub^\dR\ar[d]\\
	    \Mell^\mathrm{ord} \ar[r] \ar@{-->}[ur] & \Mell, }$$ where
	    $\Mell^\mathrm{ord}$ denotes the moduli stack of ordinary elliptic
	    curves. This existence of this dotted map is well-known in
	    arithmetic geometry: it is the statement that the Frobenius (which
    exists for ordinary elliptic curves via quotienting out by the canonical
    subgroup) splits the Hodge filtration (see \cite[Section
    A2.3]{katz-p-adic}).  \end{remark} Finally, we give the proof of Theorem
    \ref{hodge-theorem}.  \begin{proof}[Proof of Theorem \ref{hodge-theorem}] We
	shall implicitly $2$-localize everywhere; this is sufficient by Remark
	\ref{2-local}.  We begin by calculating $\H^\ast(\Mcub^\dR; g^\ast
	\omega^{\otimes \ast}) = \Ext_\Sigma(D, D)$, where $(D,\Sigma)$ is the
	Hopf algebroid $(\Z[a_1, a_2, a_3, a_4, a_6], D[s,t])$. Following
	\cite[Chapter III]{silverman}, define quantities $$b_2 = a_1^2 + 4a_2, \
	b_4 = 2a_4 + a_1 a_3, \ b_6 = a_3^2 + 4a_6, \ b_8 = a_1^2 a_6 + 4a_2 a_6
	- a_1 a_3 a_4 + a_2 a_3^2 - a_4^2.$$ Notice that $b_2 b_6 - b_4^2 =
	4b_8$ and that $1728 \Delta = c_4^2 - c_6^2$ where $c_4$, $c_6$, and
	$\Delta$ are as in the theorem statement.

    Let $I$ denote the ideal $(2, a_1, a_3, a_4)$, and define a Hopf algebroid
    $$(\ol{D}, \ol{\Sigma}) = (D/I, \Sigma/I) = (\FF_2[a_2, a_6],
    \ol{D}[s,t]).$$
    The right unit sends
    $$a_2\mapsto a_2 + s^2, \ a_6\mapsto a_6 + t^2.$$
    Then there is a Bockstein spectral sequence
    \begin{equation}\label{bockstein}
	E_1^{p,q,n} = \Ext^{p,n}_{\ol{\Sigma}}(\ol{D}, \Sym^q_{\ol{D}}(I/I^2))
	\Rightarrow \Ext^{p,n}_\Sigma(D, D),
    \end{equation}
    with $d_r:E_r^{p,q,n} \to E_r^{p+1,q+r,n}$. Now, $I/I^2 =
    \ol{D}\otimes_{\FF_2} V$, with $V = \FF_2\{\ol{a}_0, \ol{a}_1, \ol{a}_3,
    \ol{a}_4\}$ where $\ol{a}_0$, $\ol{a}_1$, $\ol{a}_3$, and $\ol{a}_4$
    represent $2$, $a_1$, $a_3$, and $a_4$ respectively.  The comodule structure
    $I/I^2 \to I/I^2\otimes_{\ol{D}} \ol{\Sigma}$ sends
    \begin{equation}\label{alpha-2}
	\ol{a}_0 \mapsto \ol{a}_0, \ \ol{a}_1 \mapsto \ol{a}_1 + \ol{a}_0 s, \
	\ol{a}_3 \mapsto \ol{a}_3 + \ol{a}_0 t, \
	\ol{a}_4 \mapsto \ol{a}_4 + \ol{a}_3 s + \ol{a}_1 t + \ol{a}_0 st.
    \end{equation}
    There is a map $\ol{D} \to \FF_2$ induced by sending $a_2$ and $a_6$
    to zero, and so we obtain a Hopf algebroid $(\FF_2, C)$ with
    $$C = \FF_2\otimes_{\ol{D}} \ol{\Sigma} \otimes_{\ol{D}} \FF_2 \cong
    \FF_2[a_2, a_6, s, t]/(a_2, a_6, \eta_R(a_2), \eta_R(a_6)) \cong
    \FF_2[s,t]/(s^2, t^2) = E(s,t).$$
    To emphasize the connection to homotopy theory, we write $h_1$ for $s$ and
    $h_{21}$ for $t$.  Now, the map $\ol{D}\to \FF_2\otimes_{\ol{D}} \ol{\Sigma}
    = \FF_2[h_1,h_{21}]$ given by sending $a_2$ to $h_1^2$ and $a_6$ to
    $h_{21}^2$ is faithfully flat, and defines an isomorphism $(\ol{D},
    \ol{\Sigma}) \to (\FF_2, C)$ of Hopf algebroids. Moreover, the Hopf
    algebroid $(\FF_2, C)$ presents the ({graded}) stack $B\alpha_2\times
    B\alpha_2$ over $\spec(\FF_2)$. It follows that
    $$\Ext^{p,n}_{\ol{\Sigma}}(\ol{D}, \Sym^q_{\ol{D}}(I/I^2)) \cong
    \Ext^{p,n}_C(\FF_2, \Sym^q_{\FF_2}(V)) = \H^{p,n}(B\alpha_2\times B\alpha_2;
    \Sym^q(V)).$$
    The comodule structure on $I/I^2$ appearing in Equation \eqref{alpha-2} is a
    representation of $\alpha_2\times \alpha_2$ on $V$, and so:
    \begin{equation}\label{regular-fixed}
	\Ext^{0,\ast}_C(\FF_2, \Sym^\ast_{\FF_2}(V)) = V^{\alpha_2\times
	\alpha_2} = \FF_2[\ol{a}_0, \ol{a}_1^2, \ol{a}_0 \ol{a}_4 + \ol{a}_1
	\ol{a}_3, \ol{a}_3^2, \ol{a}_4^2];
    \end{equation}
    indeed, these are the invariants under the $\alpha_2\times \alpha_2$-action
    on $V$. Moreover, by looking at the expressions for $b_2$, $b_4$, $b_6$, and
    $b_8$, we find that they are represented by $\ol{a}_1^2$, $\ol{a}_0 \ol{a}_4
    + \ol{a}_1 \ol{a}_3$, $\ol{a}_3^2$, and $\ol{a}_4^2$, respectively; in
    particular, all of the generators of $V^{\alpha_2\times \alpha_2}$ are
    permanent cycles in the Bockstein spectral sequence. Moreover,
    $\H^{\ast,\ast}(B\alpha_2\times B\alpha_2; \Sym^q(V)) \cong \FF_2$ for
    $q\geq 1$, and for $q=0$ we have $\H^{\ast,\ast}(B\alpha_2\times B\alpha_2;
    \Sym^0(V)) \cong \FF_2[h_1, h_{21}]$, where $h_1 = [s]$ and $h_{21} = [t]$.
    Together with Equation \eqref{regular-fixed}, we obtain a complete
    calculation of the $E_1$-page of the Bockstein spectral sequence
    \eqref{bockstein}.

    The right unit in Equation \eqref{right-unit} gives Bockstein differentials
    $a_1\mapsto 2h_1$ and $a_3 \mapsto 2h_{21}$, which correspond to $2\eta$ and
    $2\sigma_1$ being null, respectively. Moreover, 
    $$d(a_4^2) = a_3^2 s^2 + a_1^2 t^2 + 4s^2 t^2 + 4a_1st^2 - 2a_1 a_3 st -
    4a_3 s^2 t$$
    in the cobar complex, and
    $$s^2 b_6 + t^2 b_2 = a_3^2 s^2 + a_1^2 t^2 + 4a_6 s^2 + 4a_2 t^2.$$
    Since $a_2 = 0$ and $a_6 = 0$, we have $a_1 s = -s^2$ and $a_3 t = -t^2$.
    This produces a differential with target $h_1^2 b_6 + h_{21}^2 b_2$.
    Combining together all of these facts gives Equation \eqref{e2-anss-hodge}
    as the cohomology of the moduli stack $\Mcub^\dR$. As in Corollary
    \ref{nu-cor}, this is the $E_2$-page of the Adams--Novikov spectral sequence
    for $\tmf \mmod\nu$. We now calculate the Adams--Novikov differentials. The
    $15$-skeleton of $A$ is as shown in Figure \ref{A-15-skeleton}. 

    We know that $\eta^3 = 4\nu$ vanishes in $\pi_\ast A$, so $h_1^3$ must die
    in the Adams--Novikov spectral sequence for $\tmf \mmod\nu$. There is only
    one possibility, namely the $d_3$-differential $d_3(b_2) = h_1^3$. (Note
    that this differential already exists in the Adams--Novikov spectral
    sequence for $A$, where $b_2$ is represented by $v_1^2$, i.e., the class
    $[y_2]$ in the cobar complex (via Proposition \ref{A-bp}).) As a
    consequence, both $\eta$ and $\sigma_1$ are permanent cycles in the
    Adams--Novikov spectral sequence for $A$.

    Next, we know from Proposition \ref{A-bp} that $\sigma_1$ is detected in the
    Adams--Novikov spectral sequence for $A$ by $h_{21}$. Since $\sigma_1\in
    \langle \eta, \nu, 1_A\rangle$, one has that $\eta^2 \sigma_1 = 0$ in
    $\pi_\ast A$, so $h_1^2 h_{12}$ must die in $\tmf \mmod\nu$. There is no
    possibility other than $d_3(b_4)$ for a differential to kill $h_1^2 h_{21}$
    (except for a $d_3$-differential on $b_2^2$, but this cannot kill $h_1^2
    h_{12}$).
    Note that $h_1 h_{12}$ is a permanent cycle and represents $\eta \sigma_1$. 

    For the third differential, note that since $\sigma_1\in \langle \eta, \nu,
    1_A\rangle$, we have $\eta \sigma_1^2 = 0$. There is only one possibility,
    namely $d_3(b_6) = h_1 h_{21}^2$. (Since there is a $d_3$-differential
    $d_3(b_2) = h_1^3$, we have $d_3(h_{21}^2 b_2) = h_1^3 h_{21}^2$. But there
    is a relation $h_1^2 b_6 = h_{21}^2 b_2$, so $d_3(h_1^2 b_6) = h_1^3
    h_{21}^2$, which also forces $d_3(b_6) = h_1 h_{21}^2$.) Note that
    $h_{12}^2$ is a permanent cycle and represents $\sigma_1^2$.

    Finally, $\sigma_1^3 = 0$ in $\pi_\ast A$ (and therefore in $\tmf
    \mmod\nu$). It follows that the element $h_{21}^3$ must be the target of a
    differential in the Adams--Novikov spectral sequence for $\tmf \mmod\nu$.
    The only possibilities are a $d_3$-differential on $b_4^2$, $b_2 b_6$, or
    $b_8$. Only $b_8$ can kill $h_{21}^3$, and we conclude the
    $d_3$-differential $d_3(b_8) = h_{21}^3$.
    At this point, there are no more possibilities for differentials
    in the descent/Adams--Novikov spectral sequence for $\tmf \mmod\nu$, and the
    spectral sequence collapses at the $E_4$-page.

    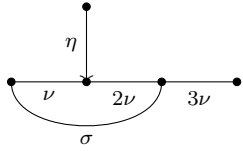
\begin{figure}
        \begin{tikzpicture}[scale=1]
            \draw [fill] (0, 0) circle [radius=0.05];
            \draw [fill] (1, 0) circle [radius=0.05];
            \draw [fill] (1, 1) circle [radius=0.05];
            \draw [fill] (2, 0) circle [radius=0.05];
            \draw [fill] (3, 0) circle [radius=0.05];
    
	    \draw (0,0) to node[below] {\footnotesize{$\nu$}} (1,0);
	    \draw (1,0) to node[below] {\footnotesize{$2\nu$}} (2,0);
	    \draw (2,0) to node[below] {\footnotesize{$3\nu$}} (3,0);

	    \draw [->] (1,1) to node[left] {\footnotesize{$\eta$}} (1,0);
    
	    \draw (0,0) to[out=-90,in=-90] node[below] {\footnotesize{$\sigma$}}
	    (2,0);
        \end{tikzpicture}
	\caption{$15$-skeleton of $A$ at the prime $2$ shown horizontally, with
	$0$-cell on the left. The element $\sigma_1$ is shown.}
        \label{A-15-skeleton}
    \end{figure}
\end{proof}

\bibliographystyle{alpha}
\bibliography{main}
\end{document}